\documentclass[twoside,bezier]{amsart}

\usepackage{epsfig}
\usepackage{amsmath,amsthm,amsfonts,amssymb,amscd,euscript}

\input{epsf}
\newcommand{\filebegin}{\begin{document}}
\newcommand{\fileend}{\end{document}}
\makeatother
\def\thefootnote{}
\newcommand{\lo}{\longrightarrow}
\newcommand{\NMM}{\hspace*{2mm}}

\renewcommand{\baselinestretch}{1.1}

\input{epsf}
\renewcommand{\baselinestretch}{1.1}
\def\n{\noindent}%

\numberwithin{equation}{section}
\def\mapdown#1{\Big\downarrow\rlap
{$\vcenter{\hbox{$\scriptstyle#1$}}$}}
\usepackage{pifont}
\newcommand{\cmark}{\ding{51}}
\newcommand{\xmark}{\ding{55}}
\newcommand{\lr}{\longrightarrow}
\newcommand{\dis}{\di^{\circledast}}
\newcommand{\fo}{\mathcal{F}}
\newcommand{\frs}[2]{\hbox{$\ds \frac{#1}{#2}$}}
\newcommand{\op}{\overline{\parzial}}
\newcommand{\pa}{\partial}
\newcommand{\paa}[2]{\partial_{#1#2}}
\newcommand{\paaa}[2]{\partial_{#1#2^2}}
\newcommand{\paaaa}[3]{\partial_{#1#2#3}}
\newcommand{\ot}{\otimes}
\newcommand{\om}{\omega}
\newcommand{\hs}[1]{\hspace{#1cm}}
\newcommand{\ric}[4]{#1\omega_{#2}\circ\omega_{#3}\otimes\omega_{#4}}
\newcommand{\rics}[4]{#1\omega_{#2}\ot\omega_{#3}\otimes\omega_{#4}}
\theoremstyle{plain}
\newtheorem{theorem}{Theorem}[section]
\newtheorem*{theorem*}{Theorem}
\newtheorem{definition}[theorem]{Definition}
\newtheorem{lemma}[theorem]{Lemma}
\newtheorem{remark}{Remark}
\newtheorem{no}[theorem]{Notation}
\newtheorem{prop}[theorem]{Proposition}
\newtheorem{cor}[theorem]{Corollary}
\newtheorem{rem}[theorem]{Remark}
\newtheorem{ex}[theorem]{Example}
\newtheorem*{mt*}{Main Theorem}
\DeclareMathOperator{\di}{d}
\newcommand\C{{\mathbb C}}
\newcommand\f{\mathcal{F}}
\newcommand\g{{\mathfrak{g}}}
\newcommand\h{{\mathfrak{h}}}
\newcommand\ka{{\mathfrak{k}}}
\renewcommand\r{{\mathfrak{r}}}
\newcommand\su{{\mathfrak{su}}}
\renewcommand\k{{\kappa}}
\renewcommand\l{{\lambda}}
\newcommand\m{{\mathfrak{m}}}
\renewcommand\O{{\omega}}
\renewcommand\t{{\theta}}
\newcommand\ebar{{\bar{\varepsilon}}}
\newcommand\R{{\mathbb R}}
\newcommand\Z{{\mathbb Z}}
\newcommand\s{{\mathbb S}}
\newcommand{\de}[2]{\frac{\partial #1}{\partial #2}}
\newcommand\w{\wedge}
\newcommand{\ov}[1]{\overline{ #1}}
\newcommand{\Tk}{\mathcal{T}_{\omega}}
\newcommand{\ovp}{\overline{\partial}}
\newcommand\ad{{\rm ad}}
\newcommand{\D}{\mathcal D}
\newcommand{\A}{\mathcal A}
\newcommand{\B}{\mathcal B}
\newcommand{\G}[3]{\Gamma_{#1#2}^{#3}}
\renewcommand{\theequation}{\thesection.\arabic{equation}}
\begin{document}
\vspace*{2cm}
\begin{center}
{\bf\large Harmonicity of left-invariant vector fields on Einstein Lorentzian Lie groups}
 \\[0.5cm]
{Yadollah Aryanejad\\ 
Department of Mathematics Payame noor University\\P.O. Box 19395-3697 Tehran Iran\\
y.aryanejad@pnu.ac.ir} \\[2mm]

\end{center}
\vspace*{0.5cm}
\begin{quotation}
\noindent
{\footnotesize
{\sc Abstract.}
We consider four-dimensional Lie groups equipped with
 left-invariant Einstein Lorentzian metrics. The harmonicity properties
of left-invariant vector fields on these spaces are determined. In some cases, all these vector fields are critical points for the energy functional
restricted to vector fields. Left-invariant vector fields defining harmonic maps are also classified, and the energy of these vector
fields is explicitly calculated.}
\end{quotation}
\n \textbf{2000 Mathematics subject classification: } 53C50, 53C15, 53C25.

\markboth
{Y.Aryanejad}
{Harmonicity of left-invariant vector fields on Einstein Lorentzian Lie groups}

\section{Introduction}
\label{}
Symmetries of the mathematical models play an important role in applied sciences (see \cite{iran1}, \cite{iran2} and \cite{iran3}).
 In \cite{b1} it is proved that a (simply-connected) four-dimensional homogeneous Riemannian manifold is either symmetric or isometric to a Lie group equipped with a left-invariant Riemannian metric. Indeed, the class of n-dimensional simply connected Lorentzian Lie groups (respectively ,Lorentzian Lie algebras) coincides with the class of the Riemannian ones. Using this fact, four-dimensional Einstein Lorentzian Lie groups have been classified \cite{c3}. 
On the other hand, investigating critical points of the energy associated to vector fields is an interesting problem from
different points of view.
In Riemannian settings, it has been proved that critical points of the energy functional $E:\mathfrak{X}(M)\rightarrow R$,
 restricted to maps defined by vector fields, are parallel vector fields (see \cite{n}, \cite{i} and \cite{g3}). Moreover, Gil-Medrano \cite{g3} studied when $V$ is a harmonic map. So, it is natural to 
determine the harmonicity properties
of vector fields on four-dimensional Lorentzian Einstein Lie groups.\\
A Riemannian manifold admitting a parallel vector field is locally reducible, and the same is true for a
pseudo-Riemannian manifold admitting an either space-like or time-like parallel vector field. This leads us to consider different situations, where some interesting types of non-parallel vector fields can be characterized in terms of harmonicity properties (see \cite{d1}, \cite{i} and \cite{n}).\\
Let $(M,g)$ be a compact pseudo-Riemannian manifold and $g_s$ the Sasaki metric
on the tangent bundle $TM$. The energy of a smooth vector field $V:(M,g)\longrightarrow
(TM,g^s)$ on $M$ is;
\begin{equation}\label{enr}
E(V)=\dfrac{n}{2}vol (M,g)+\dfrac{1}{2}\int_M ||\nabla V||^2dv
\end{equation}
(assuming M compact; in the non-compact case, one works over relatively compact domains see \cite{c1}). If $V:(M,g)\longrightarrow (TM,g^s)$ is a critical point for the energy functional, then $V$ is said to define a harmonic map. The Euler-Lagrange equations characterize vector fields $V$ defining harmonic maps as the ones whose tension field $\t(V)=tr(\nabla^2V)$ vanishes.
 Consequently,
$V$ defines a harmonic map from $(M,g)$ to $(TM,g^s)$ if and only if
\begin{equation}\label{hor}
 tr[R(\nabla_. V,V)_.]=0, \quad  \nabla^*\nabla V=0,
\end{equation}
where with respect to a pseudo-orthonormal local frame $\lbrace e_1,...,e_n\rbrace$ on $(M,g)$, with $\varepsilon_i=g(e_i,e_i)=\pm1$ for all indices i,
one has
\begin{center}
$ \nabla^*\nabla V=\sum_i \varepsilon_i( \nabla_{e_i}\nabla_{e_i} V-\nabla_{\nabla_{e_i}e_i}V)$.   
\end{center}
A smooth vector field V is said to be a harmonic
section if it is a critical point of $E^v(V)=(1/2)\int_M||\nabla V||^2dv$, where $E^v$ is the vertical energy. The corresponding Euler-Lagrange equations are given by
\begin{equation}
 \nabla^*\nabla V=0.
\end{equation}
Let $\mathfrak{X}^{\rho}(M) =\lbrace V\in \mathfrak{X}(M): ||V||^2=\rho^2 \rbrace$ and $\rho\neq 0$. Then, one can consider vector fields $ V\in \mathfrak{X}^{\rho}(M)$
which are critical points for the energy functional $E
|_{\mathfrak{X}^{\rho}(M)}$, restricted to vector fields of the same constant length. 
The
Euler-Lagrange equations of this variational condition are given by
\begin{equation}\label{hor1}
\nabla^*\nabla V\quad is\quad collinear\quad to\quad V.   
\end{equation}
In the non-compact case, the condition \eqref{hor1} is taken as a definition
of critical points for the energy functional under the assumption $\rho\neq 0$, that is, if $V$ is not light-like. 
If $\rho=0$, then \eqref{hor1} is still a sufficient condition so that $V$ is a critical point for the energy functional $E|_{\mathfrak{X}^{0}(M)}$, restricted to light-like vector fields (\cite{c1}, Theorem 26).\\
Following \cite{c3}, four-dimensional Einstein Lorentzian Lie groups are classified into four types, denoted by $\rm(a1)$, $\rm(a2)$, $\rm(c1)$ and $\rm(c2)$. In the present paper using a case-by-case argument we shall completely investigat the harmonicity of vector fields on these spaces. \\
The paper is organized as follows.
In Section 2, we recall basic properties of Einstein Lorentzian Lie algebras, as described in \cite{c3}. Harmonicity properties of vector fields on four-dimensional Einstein Lorentzian Lie group of types $\rm(a1)$, $\rm(a2)$, $\rm(c1)$ and $\rm(c2)$ will be investigated
in Sections 3-6, respectively. Finally, the energy and the minimality of all these vector fields are explicitly calculated in Section 7.

 \section{Einstein Lorentzian Lie groups}
Let $(G,g)$ be a four-dimensional Lorentzian Lie group. Following \cite{c3}, the Lie algebra $\g$ of $G$ is a semi-direct product $\r\ltimes \g _3$, where $\r =span\lbrace e_4\rbrace$ acts on $\g _3 =span\lbrace e_1,e_2,e_3 \rbrace $, and the Lorentzian inner product on $\g$ is described by
$$ 
 (a)\quad   \left( \begin{array}{cccc}
   1 & 0 & 0 & 0  \\
    0 & 1 & 0 & 0  \\
   0 & 0 & -1 & 0  \\
   0 & 0 & 0 & 1
 \end{array}  \right),
\quad\quad (c)\quad  \left( \begin{array}{cccc}
   1 & 0 & 0 & 0  \\
    0 & 1 & 0 & 0  \\
   0 & 0 & 0 & 1  \\
   0 & 0 & 1 & 0
 \end{array}  \right).\\
 $$
In 2013 Calvaruso and Zaeim \cite{c3} obtained the following result:
\begin{theorem}\label{cal33}
Let $G$ be a four-dimensional simply connected Lie group. If $g$ is a left-invariant Lorentzian Einstein metric on $G$, then the Lie algebra $\g$ of $G$ is isometric to $\g=\r\ltimes \g _3$, where $\g _3 =span\lbrace e_1,e_2,e_3 \rbrace $ and $\r =span\lbrace e_4\rbrace$, and one of the following cases occurs.\\
 ${\quad \rm (a)} $ $\lbrace e_i\rbrace _{i=1}^4$ is a pseudo-orthonormal basis, with $e_3$ time-like. In this case, $G$ is isometric to one of the following semi-direct products $\R \ltimes G_3$:\\
 ${\quad \rm (a1)} $ $\R\ltimes H$, where $H$ is the Heisenberg group and $\g$ is described by one of the following sets of conditions:
 \begin{itemize}
\item[(1)] $
  [e_1,e_2]=\epsilon A e_1,  [e_1,e_3]=A e_1, [e_1,e_4]=\delta A e_1,  [e_3,e_4]=-2A \delta (\epsilon e_2-e_3),$
\item[(2)] $[e_1,e_2]=\frac{\epsilon \sqrt{A^2-B^2}}{2} e_1, [e_1,e_3]=-\frac{\epsilon\delta \sqrt{A^2-B^2}}{2} e_1,
 [e_1,e_4]=\frac{\delta A+B}{2} e_1,  [e_2,e_4]=B (e_2+\delta e_3),  [e_3,e_4]= A(e_2+\delta e_3),$
\item[(3)] $[e_1,e_2]=\frac{\epsilon A\sqrt{A^2-B^2}}{B} e_1, [e_1,e_3]=\epsilon \sqrt{A^2-B^2} e_1, [e_2,e_4]=B e_2-A e_3,  [e_3,e_4]=A e_2-\frac{A^2}{B}e_3,
$
\item[(4)] $[e_1,e_2]=\epsilon \sqrt{A^2-B^2} e_1+Be_2,  [e_3,e_4]=A e_3,
$
\end{itemize}
 ${\quad \rm (a2)} $ $\R\ltimes \R^3$, where $\g$ is described by one of the following sets of conditions:
  \begin{itemize}
\item[(5)] $
 [e_1,e_4]=-(A+B) e_1, [e_2,e_4]=Be_2-\epsilon \sqrt{A^2+AB+B^2} e_3,  [e_3,e_4]=\epsilon \sqrt{A^2+AB+B^2} e_2+A e_3,
$
\item[(6)] $ [e_1,e_4]=-2A e_1, [e_2,e_4]=-5Ae_2+6\epsilon A e_3,  [e_3,e_4]=A e_3,
$
\item[(7)] $[e_1,e_4]=A e_1, [e_2,e_4]=A e_2+B e_3,  [e_3,e_4]=Be_2+A e_3,
$
\item[(8)] $[e_1,e_4]=\epsilon \frac{A+B}{3} e_1, [e_2,e_4]=\epsilon\frac{5B-A}{6} e_2+B e_3, [e_3,e_4]=Ae_2+\epsilon\frac{5A-B}{6} e_3,
$
\item[(9)] $ [e_1,e_4]= \frac{5A}{2} e_1+3\epsilon Ae_3, [e_2,e_4]=A e_2,  [e_3,e_4]=-\frac{A}{2} e_3,
$
\item[(10)] $ [e_1,e_4]=A e_1+\epsilon\sqrt{B^2-A^2-C^2-AC}e_2 , [e_2,e_4]=\epsilon\sqrt{B^2-A^2-C^2-AC}e_1-(A+C) e_2-Be_3,  [e_3,e_4]=B e_2+C e_3,
$
\item[(11)] $ [e_1,e_4]=- \frac{2\epsilon \sqrt{2}A}{3} e_1+\delta Ae_3, [e_2,e_4]=\frac{\epsilon \sqrt{2}A}{3} e_2,  [e_3,e_4]=A e_2-\frac{\epsilon \sqrt{2}A}{6} e_3,
$
\end{itemize}
 ${\quad \rm (c)} $ $\lbrace e_i\rbrace _{i=1}^4$  is a basis, with the inner product $g$ on $\g$ completely determined by $g(e_1,e_1)=g(e_2,e_2)=g(e_3,e_4)=g(e_4,e_3)=1$ and $g(e_i,e_j)=0$ otherwise. In this case, $G$ is isometric to one of the following semi-direct products $\R\ltimes G_3$:\\
 ${\quad \rm (c1)} $ $\R\ltimes H$,where $\g$ is described by one of the following sets of conditions
 \begin{itemize}
\item[(12)] $[e_1,e_2]=\epsilon (A+B) e_3, [e_1,e_4]=Ce_1+Be_2+De_3, [e_2,e_4]=B e_1+Ee_3, [e_3,e_4]=Ce_3,$
 \item[(13)] $[e_1,e_2]=B e_3, [e_1,e_4]=\frac{(C+D)^2-B^2}{4A}e_1+De_2+Fe_3, [e_2,e_4]=C e_1+Ae_2+Ee_3,  [e_3,e_4]=\frac{(C+D)^2-B^2+4A^2}{4A}e_3,$
 \item[(14)] $  [e_1,e_2]=\epsilon\sqrt{((A+D)^2+4B^2)} e_3, [e_1,e_4]=-Be_1+De_2+Ee_3, [e_2,e_4]=A e_1+Be_2+Ce_3,$
\end{itemize}
 ${\quad \rm (c2)} $ $\R\ltimes \R^3$, where $\g$ is described by one of the following sets of conditions:
  \begin{itemize}
\item[(15)] $ [e_1,e_2]=A e_2+Be_3, [e_2,e_4]=-A e_1+Ce_3,$
 \item[(16)]  $ [e_1,e_4]=A e_1+Be_2+Ce_3, [e_2,e_4]=D e_1+Ee_2+Fe_3, [e_3,e_4]=\frac{(B+D)^2+2(A^2+E^2)}{2(E+A)}e_3$
\end{itemize}
In all the cases listed above, $\epsilon=\pm1$, $\delta=\pm 1$ and $A, B, C, D$ are real constants. 
\end{theorem}
\section{Harmonicity of vector fields: type $\rm(a1)$}
All four-dimensional simply connected Einstein Lorentzian Lie groups of  type $\rm(a1)$ are symmetric \cite{c3} and the study of harmonic invariant vector fields on these spaces would be natural and interesting.
The main purpose of this section is to investigat the harmonicity properties
of left-invariant vector fields on four-dimensional Lorentzian Lie group of type $\rm(a1)$.
The following notation is necessary.
\begin{no}\label{no}
Let $\tilde{\mathfrak{X}}^{\rho}(M)$ denote the set of all vector fields $ V\in \mathfrak{X}^{\rho}(M)$,
which are critical points for the energy functional $E
|_{\mathfrak{X}^{\rho}(M)}$, restricted to vector fields of the same constant length. Remember that ${\rho}$  is not necessarily the same for different cases. 
\end{no}
Let $(G,g)$ be a four-dimensional Lorentzian Lie group of type $\rm(a1)$ and $\lbrace e_i\rbrace _{i=1}^4$ a pseudo-orthonormal basis, with $e_3$ time-like. Under these assumptions,
we prove the following result.
\begin{theorem}\label{hor455}
Let $\g$ be the Lie algebra of $G$ and $V=ae_1+be_2+ce_3+de_4\in \g$ a left-invariant vector field on $G$ for some real constants $a,b,c,d$. For the
different cases $(1)-(4)$ of type $\rm(a1)$, we have:
 \begin{itemize}
 \item[(1)]: $V \in \tilde{\mathfrak{X}}^{\rho}(G)$ if and only if $V=c(e_2-e_3-e_4)$, that is, $b=-c=-d$. In this case $\epsilon=1, \quad \nabla^*\nabla V=3A^2V.$
 \item[(2)]: $V \in \tilde{\mathfrak{X}}^{\rho}(G)$ if and only if $V=c(e_2+e_3-e_4)$, that is, $b=c=-d$. In this case $\epsilon=-1$, $\delta=1,\quad \nabla^*\nabla V=-\frac{3}{4}(A+B)^2V$.
 \item[(3)]: $V \in \tilde{\mathfrak{X}}^{\rho}(G)$, in this case,  $\nabla^*\nabla V=-\frac{(A^2-B^2)^2}{B^2}V$.  
 \item[(4)]: $V \in \tilde{\mathfrak{X}}^{\rho}(G)$  if and only if $a=b=0$, in this case $\nabla^*\nabla V=-A^2V$ or $c=d=0$, in this case $\nabla^*\nabla V=(B^2-A^2)V$.
\end{itemize} 
\end{theorem}
\begin{proof}
The above statement is obtained from a case-by-case argument.
 As an example, we
report the details for case $(3)$ here.
Let $V\in \g$ be a critical point for the energy functional. 
The components of the Levi-Civita connection are the following: 
\begin{eqnarray}\label{con1}
\begin{array}{cr}
\nabla_{e_1}e_1=-\frac{\epsilon A\sqrt{A^2-B^2}}{B}e_2+\epsilon\sqrt{A^2-B^2}e_3,\quad 
\nabla_{e_1}e_2=\frac{\epsilon A\sqrt{A^2-B^2}}{B}e_1\\
 \nabla_{e_1}e_3=\epsilon\sqrt{A^2-B^2}e_1\quad 
\nabla_{e_2}e_2=-Be_4,\quad 
\nabla_{e_2}e_3=-Ae_4,\\
\nabla_{e_2}e_4=Be_2-Ae_3,\hspace*{23mm} 
\nabla_{e_3}e_2=-Ae_4,\\
\nabla_{e_3}e_3=-\frac{A^2}{B}e_4,\hspace*{23mm}\nabla_{e_3}e_4=Ae_2-\frac{A^2}{B}e_3,
\end{array}
\end{eqnarray}
while $\nabla_{e_i}e_j=0$ in the remaining cases.\\
From \eqref{con1} we obtain
\begin{eqnarray}\label{con2}
\begin{array}{cr}
\nabla_{e_1}V=\epsilon\sqrt{A^2-B^2}(\frac{cB+bA}{b}e_1-\frac{aA}{B}e_2+ae_3),\\ \nabla_{e_2}V=dBe_2-dAe_3-(cA+bB)e_4,\hspace*{18mm}\nabla_{e_4}V=0\\
\nabla_{e_3}V=dAe_2-\frac{dA^2}{B}e_3-\frac{A(cA+bB)}{B}e_4.
\end{array}
\end{eqnarray}
Clearly, there are no parallel vector fields $V\neq 0$ in $\g$.
We can now calculate $\nabla_{e_i}\nabla_{e_i}V$ and $\nabla_{\nabla_{e_i}e_i}V$ for all indices i and we find
\begin{eqnarray}
\begin{array}{cr}\label{con3}
\nabla_{e_1}\nabla_{e_1}V=\frac{-(A^2-B^2)}{B^2}(a(A^2-B^2)e_1+(cB+bA)(Ae_2-Be_3)),\\ \nabla_{e_2}\nabla_{e_2}V=-(cB+bA)(Be_2-Ae_3)+d(A^2-B^2)e_4,\\ \nabla_{e_3}\nabla_{e_3}V=\frac{-A^2}{B^2}((cB+bA)(Be_2-Ae_3)+d(A^2-B^2)e_4),\quad\nabla_{e_4}\nabla_{e_4}V=0,\\
\nabla_{\nabla_{e_1}e_1}V=0,\quad\nabla_{\nabla_{e_3}e_3}V=0
\\
\nabla_{\nabla_{e_2}e_2}V=0,
  \quad\quad
\nabla_{\nabla_{e_4}e_4}V=0.\quad
\end{array}
\end{eqnarray}
Then, we get
\begin{eqnarray}
\begin{array}{cc}\nonumber
 \nabla^*\nabla V=& \hspace{-2cm}\sum_i \varepsilon_i( \nabla_{e_i}\nabla_{e_i} V-\nabla_{\nabla_{e_i}e_i}V)=\frac{-(A^2-B^2)}{B^2}(a(A^2-B^2)e_1+\\ &(cB+bA)(Ae_2-Be_3))-(cB+bA)(Be_2-Ae_3)+d(A^2-B^2)e_4-\\ &\hspace{-1.5cm}(\frac{-A^2}{B^2}((cB+bA)(Be_2-Ae_3)+d(A^2-B^2)e_4))=-\frac{(A^2-B^2)^2}{B^2}V.  
\end{array}
\end{eqnarray}
\end{proof}
As the definitions already show, $V$ is harmonic if $ \nabla^*\nabla V=0$ and $V$ defines a harmonic map if and only if
\begin{center}
$ tr[R(\nabla_. V,V)_.]=0, \quad  \nabla^*\nabla V=0.$
\end{center}
For case $(3)$ in Theorem \ref{hor455}, 
$ \nabla^*\nabla V=-\frac{(A^2-B^2)^2}{B^2}V=0$ if and only if $A=\pm B$, that is, $V$ is harmonic if and only if $A=\pm B$. 
Let $R$ denote the curvature tensor of $(M,g)$, taken with the sign convention $R(X,Y ) = \nabla [X,Y]-[\nabla X, \nabla Y ]$.
Then, using \eqref{con2}, we find
\begin{center}
$R(\nabla_{e_1} V,V)e_1=\frac{\epsilon(A^2-B^2)^{3/2}}{B^3} ((A^2-B^2)a^2+(bA+cB)^2)(Ae_2-Be_3),$\\ $ \frac{A^2}{B^2}R(\nabla_{e_2} V,V)e_2=R(\nabla_{e_3} V,V)e_3=\frac{A^2(A^2-B^2)}{B^3}((A^2-B^2)d^2-(cA+bB)^2)e_4 ,\quad R(\nabla_{e_4} V,V)e_4=0$   
\end{center}
and so, when $A=\pm B$ clearly,
\begin{center}
$tr[R(\nabla_. V,V)_.]=\sum_{i} \varepsilon_i R(\nabla_{e_i} V,V)e_i=0$.   
\end{center}
Hence, $tr[R(\nabla_. V,V)_.]=0$ if and only if $A=\pm B$. 
Appling this argument for other cases of type $\rm(a1)$ proves the following classification result.
\begin{theorem}\label{horm}
Let $V$ be a critical point for the energy functional, described by the conditions $(2)$ and $(3)$ in Theorem \ref{hor455}. Then, for cases $(2)$ and $(3)$, $V$ defines a harmonic map if and only if $A=-B$ and $A=\pm B$ respectively.
\end{theorem} 
\begin{table}
\caption{Equivalent properties for the
cases $(1)-(4)$ in Theorem \ref{hor455}: type $\rm(a1)$.\label{tab}}
\begin{tabular}{|p{1cm}|p{13cm}|}
\hline
$(G,g)$ &   Equivalent properties (denoted by $\equiv$) \\
\hline
$(1)$&  $V$  is geodesic; $\hspace*{2mm}\equiv\hspace*{2mm}$  $V \in \tilde{\mathfrak{X}}^{\rho}(G)$; $\hspace*{2mm}\equiv\hspace*{2mm}$  none of these vector fields is harmonic (in particular, defines a harmonic map); $\hspace*{2mm}\equiv\hspace*{2mm}$  $V=c(e_2-e_3-e_4)$,\\
\hline
$(2)$&$V$  is geodesic; $\hspace*{2mm}\equiv\hspace*{2mm}$  $V$ is harmonic if and only if $A=-B;$ $\hspace*{2mm}\equiv\hspace*{2mm}$  $V \in \tilde{\mathfrak{X}}^{\rho}(G)$; $\hspace*{2mm}\equiv\hspace*{2mm}$  $V$ defines harmonic map if and only if $A=-B;$ $\hspace*{2mm}\equiv\hspace*{2mm}$  $V$ is Killing if and only if $A=-B$ and $d=0$; $\hspace*{2mm}\equiv\hspace*{2mm}$  $V=c(e_2+e_3-e_4)$,\\
\hline
$(3)$& $V$  is geodesic if and only if $A=\pm B$ and $b=\mp c$; $\hspace*{2mm}\equiv\hspace*{2mm}$ $V$ is harmonic if and only if $A=\pm B$; $\hspace*{2mm}\equiv\hspace*{2mm}$  $V \in \tilde{\mathfrak{X}}^{\rho}(G)$; $\hspace*{2mm}\equiv\hspace*{2mm}$  $V$ defines harmonic map if and only if $A=\pm B$; $\hspace*{2mm}\equiv\hspace*{2mm}$  $V$ is Killing if and only if $A=\pm B$, $b=\mp c$ and $d=0$,\\
\hline
$(4)$&  $V$  is geodesic if and only if $a=b=c=0$; $\hspace*{2mm}\equiv\hspace*{2mm}$ $V \in \tilde{\mathfrak{X}}^{\rho}(G)$ if and only if $a=b=0$; $\hspace*{2mm}\equiv\hspace*{2mm}$  none of these vector fields is harmonic (in particular, defines a harmonic map).\\
\hline
\end{tabular}
\end{table}
A vector field $V$ is geodesic if $\nabla_VV =0$, and is Killing if $\mathcal{L}_V g=0$, where $\mathcal{L}$ denotes
the Lie derivative. Parallel vector fields are both geodesic and Killing, and vector fields with these special geometric features often have particular harmonicity properties. A straightforward calculation proves the following main classification result.
\begin{cor}\label{hor6}
If $g$ is a left-invariant Lorentzian Einstein metric on $G$, then for the
different cases $(1)-(4)$ in Theorem \ref{hor455}, the equivalent properties for $V=ae_1+be_2+ce_3+de_4\in \g$ are listed in Table 1.
\end{cor}
\begin{rem}\label{spatially}
Recall that for a Lorentzian Lie group, a left-invariant vector field $V$ is spatially
harmonic if and only if
\begin{equation}\label{hor13456}
\tilde{X}_V=-\nabla^*\nabla V-\nabla_{V}\nabla_{V}V-div V\cdot \nabla_{V}V+(\nabla V)^t \nabla_{V}V  \quad is\quad collinear\quad to\quad V.   
\end{equation}
Clearly, conditions \eqref{hor1} and \eqref{hor13456} coincide for geodesic vector fields. Hence, the results listed in Table 1 show that for cases $(1)$ and $(2)$, $V$ is spatially harmonic and for case $(3)$, $V$ is spatially harmonic if and only if
$A=\pm B$ and $b=\mp c$. For case $(4)$, $V$ is spatially harmonic if and only if
$a=b=c=0$.
\end{rem}
\section{Harmonicity of vector fields: type $\rm(a2)$}
Let $(G,g)$ be a four-dimensional Lorentzian Lie group of type $\rm(a2)$, 
$\g$ the Lie algebra of $G$ and $V=ae_1+be_2+ce_3+de_4\in \g$ a left-invariant vector field on $G$, for some real constants $a,b,c,d$. The Lie algebra $\g$ is described by one of the sets of conditions $(5)-(11)$ in Theorem \ref{cal33}.
As an example, for case $(7)$ a direct calculation
yields that we can describe the Levi-Civita connection as follows:
  \begin{eqnarray}\label{conc12}
\begin{array}{cr}
\nabla_{e_1}e_1=-Ae_4\quad \quad \nabla_{e_1}e_4=Ae_1,\quad \quad
\nabla_{e_2}e_2=-Ae_4\quad \quad \nabla_{e_2}e_4=Ae_2,\\ 
\nabla_{e_3}e_3=Ae_4\quad \quad \nabla_{e_3}e_4=Ae_3, \quad \quad
\nabla_{e_4}e_2=-Be_3\quad \quad \nabla_{e_4}e_3=-Be_2.
\end{array}
\end{eqnarray}
Using \eqref{conc12} to calculate $\nabla_{e_i}V$ for all indices i, we get
\begin{eqnarray}\label{XX}
\begin{array}{cr}
\nabla_{e_1}V=A(de_1-ae_4),\quad\nabla_{e_2}V=A(de_2-be_4),\\
\nabla_{e_3}V=A(de_3+ce_4),\quad \nabla_{e_4}V=-B(ce_2+be_3).
\end{array}
\end{eqnarray}
Clearly, there are no parallel vector fields $V\neq 0$ in $\g$. We can now use \eqref{XX} to obtain $\nabla_{e_i}\nabla_{e_i}V$ for all indices i and we find
\begin{eqnarray}
\begin{array}{cr}\label{con3a2}
\nabla_{e_1}\nabla_{e_1}V=-A^2(ae_1+de_4),\quad\nabla_{e_2}\nabla_{e_2}V=-A^2(be_2+de_4),\\ \nabla_{e_3}\nabla_{e_3}V=A^2(ce_3+de_4),\quad\nabla_{e_4}\nabla_{e_4}V=B^2(be_2+ce_3).\\
\end{array}
\end{eqnarray}
We calculate $\nabla_{\nabla_{e_i}e_i}V$ for all indices i. We obtain
\begin{eqnarray}
\begin{array}{cr}\label{con3}
\nabla_{\nabla_{e_1}e_1}V=\nabla_{\nabla_{e_2}e_2}V=-\nabla_{\nabla_{e_3}e_3}V=AB(ce_2+be_3),\quad
\nabla_{\nabla_{e_4}e_4}V=0.
\end{array}
\end{eqnarray}
Then, we get
\begin{eqnarray}
\begin{array}{cr}\label{con33a2}
\nabla^*\nabla V=\sum_i \varepsilon_i( \nabla_{e_i}\nabla_{e_i} V-\nabla_{\nabla_{e_i}e_i}V)=\\
-A^2ae_1+((B^2-A^2)b-3ABc)e_2+((B^2-A^2)c-3ABb)e_3-3A^2de_4. 
\end{array}
\end{eqnarray}
Hence, $\nabla^*\nabla V=(B^2-A^2-3AB)V$ if and only if $b=c$ and $a=d=0$. 

Using this argument for other cases leads to the following result.
\begin{theorem}\label{hor4a2}
Let $G$ be a four-dimensional simply connected Lie group of type $\rm(a2)$ and $V=ae_1+be_2+ce_3+de_4$ a left-invariant vector field on $G$, for some real constants $a,b,c,d$. For the
different cases $(5)-(11)$ in Theorem \ref{cal33}, we have:
 \begin{itemize}
 \item[(5)]: $V \in \tilde{\mathfrak{X}}^{\rho}(G)$ if and only if $V=ae_1$, that is, $b=c=d=0$. In this case $\nabla^*\nabla V=-(A+B)^2V.$
 \item[(6)]: $V \in \tilde{\mathfrak{X}}^{\rho}(G)$ if and only if  $V=b(e_2-e_3)$, that is, $c=-b$. In this case $\epsilon=1, \quad \nabla^*\nabla V=-13A^2V$.
 \item[(7)]: $V \in \tilde{\mathfrak{X}}^{\rho}(G)$ if and only if  $V=b(e_2+e_3)$, that is, $c=b$. In this case $\nabla^*\nabla V=(B^2-A^2+3AB)V$.
 \item[(8)]: $V \in \tilde{\mathfrak{X}}^{\rho}(G)$ if and only if  $V=b(e_2-e_3)$, that is, $c=-b$. In this case $\epsilon=-1, \quad \nabla^*\nabla V=\frac{13}{36}(A+B)^2V.$
 \item[(9)]: $V \in \tilde{\mathfrak{X}}^{\rho}(G)$ if and only if  $V=a(e_1+e_3)$, that is, $c=a$. In this case $\epsilon=1, \quad \nabla^*\nabla V=-\frac{13}{4}A^2V$. 
  \item[(10)]:$V \in \tilde{\mathfrak{X}}^{\rho}(G)$ if and only if  $V=ce_3$, that is, $a=b=d=0$. In this case $ \nabla^*\nabla V=-C^2V$. 
   \item[(11)]: $V \in \tilde{\mathfrak{X}}^{\rho}(G)$ if and only if  $V=c(\sqrt{2}e_1-\frac{2}{3}\sqrt{2}e_2+e_3)$, that is, $a=\sqrt{2}c,\quad b=-\frac{2}{3}\sqrt{2}c$. In this case $\epsilon=\delta=-1, \quad \nabla^*\nabla V=\frac{5}{18}A^2V$. 
\end{itemize} 
\end{theorem}
Applying \eqref{hor} and \eqref{con33a2} to case $(7)$ in Theorem \ref{hor4a2},  
by standard calculations we
prove that $ \nabla^*\nabla V=0$ if and only if $a=d=0$ and $(B^2-A^2+3AB)=0$, that is, $A=((3-\sqrt{13})/2)B$.
Using \eqref{XX} for  $V=b(e_2+e_3)$, we find
\begin{center}
$R(\nabla_{e_2} V,V)e_2=R(\nabla_{e_3} V,V)e_3=A^2(A+2B)c^2e_4$,\\
$R(\nabla_{e_1} V,V)e_1=R(\nabla_{e_4} V,V)e_4=0$
\end{center}
and so,
\begin{center}
$tr[R(\nabla_. V,V)_.]=\sum_{i} \varepsilon_i R(\nabla_{e_i} V,V)e_i=0$.   
\end{center}
Following this argument for other cases of type $\rm(a2)$ proves the following result.
 \begin{theorem}\label{hor8a2}
Let $V=ae_1+be_2+ce_3+de_4$ be a critical point for the energy functional described by the different conditions $(5)-(11)$ in Theorem \ref{hor4a2}, then it is easy to check that
 \begin{itemize}
\item[$\bullet$] For cases $(5)$ and $(8)$, $V$ is a harmonic map  if and only if  $A=-B$.
\item[$\bullet$] For case $(7)$, $V$ is a harmonic map if and only if $A=((-3\pm \sqrt{13})/2)B$.
\item[$\bullet$] For case $(10)$, $V$ is a harmonic map if and only if  $C=B=0$.
\end{itemize}
\end{theorem}
Starting from \eqref{con3a2}, a straightforward calculation proves the following classification result.
\begin{cor}\label{hor6}
If $g$ is a left-invariant Lorentzian Einstein metric on $G$,  then for the
different cases $(5)-(11)$ in Theorem \ref{hor4a2}, the equivalent properties for $V$ are listed in Table 2.
\end{cor}
 \begin{table}
 \caption{Set of equivalent properties for vector fields: type $\rm(a2)$.\label{tab}}
\begin{tabular}{|p{1cm}|p{13cm}|}
\hline
$(G,g)$ &   Equivalent properties (denoted by $\equiv$) \\
\hline
$(5)$&$V$ is harmonic if and only if $A=-B$;$\hspace*{2mm}\equiv\hspace*{2mm}$  $V \in \tilde{\mathfrak{X}}^{\rho}(G)$ if and only if $V=ae_1$;$\hspace*{2mm}\equiv\hspace*{2mm}$  $V$ defines harmonic map if and only if $A=-B$;$\hspace*{2mm}\equiv\hspace*{2mm}$  $V$ is Killing if and only if $A=-B$ and  $V=ae_1$.\\
\hline
$(6)$& $V$  is geodesic; $\hspace*{2mm}\equiv\hspace*{2mm}$  $V \in \tilde{\mathfrak{X}}^{\rho}(G)$; $\hspace*{2mm}\equiv\hspace*{2mm}$  none of these vector fields is harmonic (in particular, defines a harmonic map); $\hspace*{2mm}\equiv\hspace*{2mm}$  $V=b(e_2-e_3)$.\\
\hline
$(7)$& $V$  is geodesic; $\hspace*{2mm}\equiv\hspace*{2mm}$  $V$ is harmonic if and only if $A=((-3\pm\sqrt{13})/2)B$; $\hspace*{2mm}\equiv\hspace*{2mm}$  $V \in \tilde{\mathfrak{X}}^{\rho}(G)$; $\hspace*{2mm}\equiv\hspace*{2mm}$  $V$ defines harmonic map if and only if $A=((-3\pm\sqrt{13})/2)B$; $\hspace*{2mm}\equiv\hspace*{2mm}$  $V$ is Killing if and only if $A=-B$; $\hspace*{2mm}\equiv\hspace*{2mm}$  $V=b(e_2+e_3)$.\\
\hline
$(8)$&$V$  is geodesic if and only if $A=-B$ and $c=-b$;$\hspace*{2mm}\equiv\hspace*{2mm}$  $V$ is harmonic if and only if $A=-B$; $\hspace*{2mm}\equiv\hspace*{2mm}$  $V \in \tilde{\mathfrak{X}}^{\rho}(G)$ if and only if  $c=-b, a=d=0$; $\hspace*{2mm}\equiv\hspace*{2mm}$  $V$ defines harmonic map if and only if $A=-B$; $\hspace*{2mm}\equiv\hspace*{2mm}$  $V$ is Killing if and only if $A=-B$ and $d=0$; $\hspace*{2mm}\equiv\hspace*{2mm}$   $V=ae_1+be_2+ce_3+de_4$.\\
\hline
$(9)$& $V$  is geodesic; $\hspace*{2mm}\equiv\hspace*{2mm}$  $V \in \tilde{\mathfrak{X}}^{\rho}(G)$; $\hspace*{2mm}\equiv\hspace*{2mm}$  none of these vector fields is harmonic (in particular, defines a harmonic map); $\hspace*{2mm}\equiv\hspace*{2mm}$  $V=a(e_1+e_3)$.\\
\hline
$(10)$& $V$ is geodesic if and only if $C=0$;$\hspace*{2mm}\equiv\hspace*{2mm}$  $V$ is harmonic if and only if $C=0$; $\hspace*{2mm}\equiv\hspace*{2mm}$  $V \in \tilde{\mathfrak{X}}^{\rho}(G)$; $\hspace*{2mm}\equiv\hspace*{2mm}$  $V$ defines harmonic map if and only if $C=0$; $\hspace*{2mm}\equiv\hspace*{2mm}$  $V$ is Killing if and only if $C=0$; $\hspace*{2mm}\equiv\hspace*{2mm}$  $V=ce_3$.\\
\hline
$(11)$& $V \in \tilde{\mathfrak{X}}^{\rho}(G)$; $\hspace*{2mm}\equiv\hspace*{2mm}$  none of these vector fields is harmonic (in particular, defines a harmonic map); $\hspace*{2mm}\equiv\hspace*{2mm}$  $V=c(\sqrt{2}e_1-\frac{2}{3}\sqrt{2}e_2+e_3)$.\\
%
\hline
\end{tabular}
\end{table}
Using Remark \ref{spatially}, the results listed in Table 2 show that for cases $(6)-(9)$, $V$ is spatially harmonic and for case $(10)$, $V$ is spatially harmonic if and only if
$C=0$.
\section{Harmonicity of vector fields: type $\rm(c1)$}
 Consider a four-dimensional simply connected Lie group $G$ of type $\rm(c1)$ and a basis $\lbrace X_i\rbrace _{i=1}^4$, with non-zero inner product $g$ on $\g$ determined by $g(X_1,X_1)=g(X_2,X_2)=g(X_3,X_4)=g(X_4,X_3)=1$.
We can construct a pseudo-orthonormal frame field $\lbrace e_1,e_2,e_3,e_4\rbrace$, putting
\begin{eqnarray}\label{frame}
\begin{array}{cr}
e_1=X_1,\quad e_2=X_2,\quad e_3=-(1/2)X_3+X_4,\quad e_4=(1/2)X_3+X_4.
\end{array}
\end{eqnarray}
Clearly, $e_3$ is time-like. A vector field $V\in \g$ is uniquely determined by its components with respect to the pseudo-orthonormal basis $\lbrace e_i \rbrace$. Hence $V=ae_1+be_2+ce_3+de_4\in \g$ is a left-invariant vector field on $G$ for some real constants $a,b,c,d$.
Notice that the (constant) norm of $V$ is given by $||V||^2=a^2+b^2-c^2+d^2$.
As an example, for case $(14)$, using \eqref{frame} we find
\begin{eqnarray}\label{conc1}
\begin{array}{cr}
\nabla_{e_1}e_1=-B(e_3-e_4),\quad 
\nabla_{e_1}e_2=\dfrac{1}{2}(A+D-\epsilon\alpha)(e_3-e_4),\\
\nabla_{e_1}e_3=\nabla_{e_1}e_4=-Be_1+\dfrac{1}{2}(A+D-\epsilon\alpha)e_2,\\
\nabla_{e_2}e_1=\dfrac{1}{2}(A+D+\epsilon\alpha)(e_3-e_4),\quad  
\nabla_{e_2}e_2=B(e_3-e_4),\\
\nabla_{e_2}e_3=\nabla_{e_2}e_4=\dfrac{1}{2}(A+D+\epsilon\alpha)e_1+Be_2,\\
\nabla_{e_4}e_1=\nabla_{e_3}e_1=\dfrac{1}{2}(A-D-\epsilon\alpha)e_2+E(e_3-e_4),
\end{array}
\end{eqnarray}
\begin{center}
$
\nabla_{e_4}e_2=\nabla_{e_3}e_2=\dfrac{1}{2}(-A+D-\epsilon\alpha)e_1+C(e_3-e_4),$\\$
\nabla_{e_4}e_3=\nabla_{e_4}e_4=\nabla_{e_3}e_3=\nabla_{e_3}e_4=Ee_1+Ce_2,$
\end{center}
where $\alpha=\sqrt{(A+D)^2+4B^2}$.
 Suppose that $u=e_3-e_4$. Then, using \eqref{conc1}, we get that $\nabla_{e_i}u=0$ for all indices i. Therefore, $u$ is a parallel light-like vector field. The existence of a parallel light-like vector field is an interesting phenomenon which has no Riemannian counterpart, and characterizes a class of pseudo-Riemannian manifolds which illustrate many of differences between Riemannian and pseudo-Riemannian settings (see for example \cite{ch1},\cite{ch2}).
Let $V =ae_1+be_2+ce_3+de_4 \in \g$ be an arbitrary left-invariant vector field. 
We can use \eqref{conc1} to calculate $\nabla_{e_i}V$ for all indices i. We get
\begin{eqnarray}\label{Xc1}
\begin{array}{cr}
\nabla_{e_1}V=(c+d)(-Be_1+\dfrac{1}{2}(D+A-\beta)e_2)\\
+(\dfrac{1}{2}(A+D-\beta)b-Ba)u,\\
\nabla_{e_2}V=(c+d)(\dfrac{1}{2}(D+A+\beta)e_1+Be_2)\\
+(\dfrac{1}{2}(A+D+\beta)a-Bb)u,\\
\nabla_{e_3}V=\nabla_{e_4}V=(E(c+d)+(\dfrac{1}{2}(D-A+\beta)b)e_1 \\
 +(C(c+d)-(\dfrac{1}{2}(D-A+\beta)a)e_2+(Cb+Ea)u,
\end{array}
\end{eqnarray}
where $\beta=\epsilon\sqrt{(A+D)^2+4B^2}$.
We can now use \eqref{Xc1} to calculate $ \nabla_{e_i}\nabla_{e_i} V$ and we find
\begin{eqnarray}
\begin{array}{cr}\label{con3c1}
\hspace{1cm}  \nabla_{e_1}\nabla_{e_1}V=-B(c+d)e_1+\frac{1}{2}(A+T-\beta)(c+d)e_2\\
+(b\frac{1}{2}(A+T-\beta)-aB)u,\\
\hspace{1.3cm} \nabla_{e_2}\nabla_{e_2}V=\frac{1}{2}(A+T+\beta)(c+d)e_1+B(c+d)e_2\\
+(a\frac{1}{2}(A+T+\beta)+bB)u,\\
\hspace{1cm} \nabla_{e_3}\nabla_{e_3}V=\nabla_{e_4}\nabla_{e_4}V=
-\frac{1}{4}(A-T-\beta)^2(e_1+e_2)\\
+\frac{1}{2}(Ca-Eb)(A-T-\beta)^2u,
\end{array}
\end{eqnarray}
Then, we get
\begin{eqnarray}
\begin{array}{cr}\label{con31c1}
\nabla_{\nabla_{e_1}e_1}V=\nabla_{\nabla_{e_2}e_2}V=0,\\
\nabla_{\nabla_{e_3}e_3}V=\nabla_{\nabla_{e_4}e_4}V=(\frac{1}{2}C(A+T+\beta)-BE)(c+d)e_1+\\ (BC+\frac{1}{2}E(A+T-\beta)(c+d)e_4+(b(BC+\frac{1}{2}E(A+T-\beta))+\\ a(\frac{1}{2}C(A+T+\beta)-BE))u,
\end{array}
\end{eqnarray}
Hence, from \eqref{con3c1} and \eqref{con31c1} we deduce
\begin{center} 
$ \nabla^*\nabla V=\sum_i \varepsilon_i( \nabla_{e_i}\nabla_{e_i} V-\nabla_{\nabla_{e_i}e_i}V)=((A+D)^2+4B^2)(c+d)u$, \\
\end{center}
that is, $\nabla^*\nabla V$ identically vanishes if and only if $c=-d.$ 
Using \eqref{Xc1}, the curvature tensor is completely determined by 
\begin{center} 
$R(\nabla_{e_1} V,V)e_1=-R(\nabla_{e_2} V,V)e_2=\dfrac{1}{2}B(c+d)^2((A+D)^2+4B^2+(D-A)\epsilon\alpha)u,$\\
$R(\nabla_{e_3} V,V)e_3=R(\nabla_{e_4} V,V)e_4=\dfrac{1}{4}B(c+d)((A+D)^2+4B^2+(D-A)\epsilon\alpha) $\\
$(-2E(c+d)-b(D-A+\epsilon \alpha))e_1-(2C(c+d)-a(D-A+\epsilon \alpha))e_2),$\\
\end{center}
where $\alpha=\sqrt{(A+D)^2+4B^2}$. Therefore
\begin{center} 
$tr[R(\nabla_. V,V)_.]=\sum_{i} \varepsilon_i R(\nabla_{e_i} V,V)e_i=0$.
\end{center}
Applying this argument for other cases we prove the following result.
 \begin{theorem}\label{hor8}
Let $G$ be a four-dimensional simply connected Lie group of type $\rm(c1)$ and $V=ae_1+be_2+ce_3+de_4$ a left-invariant vector field on $G$, for some real constants $a,b,c,d$. For the
different cases  $(12),(13)$ and $(14)$ in Theorem \ref{cal33}, $V$ is a critical point for the energy functional $E
|_{\mathfrak{X}^{\rho}(M)}$, that is, 
$V \in \tilde{\mathfrak{X}}^{\rho}(G)$ (which ${\rho}$ is not necessarily the same for the different cases). Moreover, $V$ defines a harmonic map if and only if  $V=ae_1+be_2+cu$, that is, $d=-c$.
\end{theorem}
Therefore, for all these cases, left-invariant harmonic vector fields define harmonic maps, form three-parameter families.
Also with regard to harmonicity properties of invariant vector fields,  four-dimensional simply connected Lie groups of type $\rm(c1)$ display
some particular features. The main geometrical reasons for the special behaviour
of these Lie groups are the existence of a parallel light-like vector field.
Starting from \eqref{conc1}, we can easily prove the following classification result.
\begin{prop}\label{pro2222}
Let $G$ be a four-dimensional simply connected Lie group of type $\rm(c1)$ and $V \in \g$ a left-invariant vector field on $G$. If $g$ is a left-invariant Lorentzian Einstein metric on $G$, then we have the classification listed in Table 3.
\end{prop}
\begin{table}
\caption{ Geodesic, killing and parallel vector fields on $G$: type $\rm(c1)$.\label{tab3}}
\begin{tabular}{|c|c|c|c|}
\hline
$(G,g)$ & Geodesic vector fields & Killing vector fields & parallel vector fields \\
\hline
$(12)$&  $V=be_2+cu$ &  $V=cu \Leftrightarrow C=0$ & $V=cu \Leftrightarrow C=0$\\
\hline
$(13)$&$V=ae_1-\frac{B-C-T}{2A}ae_2+cu$& $V=cu \Leftrightarrow$ & $V=cu \Leftrightarrow$\\
&& $ 4A^2=B^2-(C+T)^2$ & $ 4A^2=B^2-(C+T)^2$\\
\hline
$(14)$&$V=cu$&$V=cu$ & $V=cu$\\
\hline
\end{tabular}
\end{table}
 Comparing Proposition \ref{pro2222} with Theorem \ref{hor8}, one sees  the following main result which emphasizes once again the special role played
by the parallel vector field $u$. 
\begin{cor}\label{hor12c1}
Let $g$ be a left-invariant Lorentzian Einstein metric on $G$, then for the
different cases $(12)-(14)$, the equivalent properties for $V$ are listed in the following Table 4.
\end{cor}

\begin{table}
\caption{Set of equivalent properties for vector fields: type $\rm(c1)$. .\label{tab}}
\begin{tabular}{|p{1cm}|p{13cm}|}
\hline
$(G,g)$ &   Equivalent properties (denoted by $\equiv$) \\
\hline
$(12)$ &  $V$  is geodesic if and only if $a=0$;    $\hspace*{2mm}\equiv\hspace*{2mm}$    $V$ is harmonic; $\hspace*{2mm}\equiv\hspace*{2mm}$    $V \in \tilde{\mathfrak{X}}^{\rho}(G)$ $\hspace*{2mm}\equiv\hspace*{2mm}$    $V$ defines harmonic map; $\hspace*{2mm}\equiv\hspace*{2mm}$    $V$ is Killing if and only if $a=b=C=0$ , that is, V is collinear to $u$; $\hspace*{2mm}\equiv\hspace*{2mm}$ $V$ is parallel if and only if $a=b=C=0$ , that is, V is collinear to $u$; $\hspace*{2mm}\equiv\hspace*{2mm}$     $V=ae_1+be_2+cu$,\\
\hline
$(13)$&  $V$  is geodesic  if and only if $b=-\frac{B-C-T}{2A}a$; $\hspace*{2mm}\equiv\hspace*{2mm}$     $V$ is harmonic; $\hspace*{2mm}\equiv\hspace*{2mm}$     $V \in \tilde{\mathfrak{X}}^{\rho}(G)$ $\hspace*{2mm}\equiv\hspace*{2mm}$     $V$ defines harmonic map;  $\hspace*{2mm}\equiv\hspace*{2mm}$ $V$ is Killing if and only if $a=b=0,\quad 4A^2=B^2-(C+T)^2 $, that is, V is collinear to $u$; $\hspace*{2mm}\equiv\hspace*{2mm}$  $V$ is parallel if and only if  $a=b=0,\quad 4A^2=B^2-(C+T)^2 $; $\hspace*{2mm}\equiv\hspace*{2mm}$    $V=ae_1+be_2+cu$,\\
\hline
$(14)$&$V$ is geodesic if and only if $a=b=0$; $\hspace*{2mm}\equiv\hspace*{2mm}$    $V$ is harmonic; $\hspace*{2mm}\equiv\hspace*{2mm}$     $V \in \tilde{\mathfrak{X}}^{\rho}(G)$ $\hspace*{2mm}\equiv\hspace*{2mm}$    $V$ defines harmonic map; $\hspace*{2mm}\equiv\hspace*{2mm}$ $V$ is Killing if and only if $a=b=0,$ $\hspace*{2mm}\equiv\hspace*{2mm}$  $V$ is parallel if and only if $a=b=0$; $\hspace*{2mm}\equiv\hspace*{2mm}$    $V=ae_1+be_2+cu$,\\
\hline
\end{tabular}
\end{table}
For cases $(12)$ and $(13)$, by Theorem \ref{hor12c1}, it is easily proved that $V$ is spatially harmonic if and only if
$a=0$, $b=-\frac{B-C-T}{2A}a$, and for case $(14)$, $V$ is spatially harmonic if and only if $a=b=0$.  
\section{Harmonicity of vector fields: type $\rm(c2)$}
We start classifying left-invariant vector fields on four-dimensional simply connected Lie group $G$ of type $\rm(c2)$, proving the following result.
 \begin{theorem}\label{hor8c2}
 Let $(G,g)$ be a four-dimensional Lorentzian Lie group of type $\rm(c2)$. For cases $(15)$ and $(16)$, all vector fields in $\g$ are critical points for the energy functional
restricted to vector fields of the same length, that is, 
$V \in \tilde{\mathfrak{X}}^{\rho}(G)$ (which ${\rho}$ is not necessarily the same for the different cases).\\  Moreover, for case $(15), $ $V$ defines a harmonic map. For case $(16)$,  $V=ae_1+be_2+ce_3+de_4$ defines a harmonic map if and only if $c=-d$.
\end{theorem}
 \begin{proof}
For case $(15)$, using \eqref{frame} we deduce 
  \begin{eqnarray}\label{conc2}
\begin{array}{cr}
\nabla_{e_3}e_1=\nabla_{e_4}e_1=-Ae_2+Be_3-Be_4,\\
\nabla_{e_3}e_2=\nabla_{e_4}e_2=Ae_1+Ce_3-Ce_4,\\
\nabla_{e_3}e_3=\nabla_{e_4}e_3=\nabla_{e_3}e_4=\nabla_{e_4}e_4=Be_1+Ce_2,
\end{array}
\end{eqnarray}
while $\nabla_{e_i}e_j=0$ in the remaining cases. Let $V =ae_1+be_2+ce_3+de_4 \in \g$ be a vector field. From \eqref{conc2} we get
\begin{eqnarray}\label{Xc2}
\begin{array}{cr}
\nabla_{e_1}V=\nabla_{e_2}V=0,\\
\nabla_{e_3}V=\nabla_{e_4}V=(Ab+Bc+Bd)e_1\\+(-Aa+Cc+Cd)e_2+(Ba+Cb)e_3-(Ba+Cb)e_4\\
=(Ab+Bc+Bd)e_1+(-Aa+Cc+Cd)e_2+(Ba+Cb)u.
\end{array}
\end{eqnarray}
Thus, $u$ is a parallel vector field, where we note the special role of $u=e_3-e_4$.
We calculate $ \nabla_{e_i}\nabla_{e_i} V$ and $\nabla_{\nabla_{e_i}e_i}V$ and we find
\begin{eqnarray}
\begin{array}{cr}\label{con3c2}
\nabla_{e_1}\nabla_{e_1}V=\nabla_{e_2}\nabla_{e_2}V=0,\\ \nabla_{e_3}\nabla_{e_3}V=\nabla_{e_4}\nabla_{e_4}V=A(-Aa+Cc+Cd)e_1\\
-A(Ab+Bc+Bd)e_2+(C(-Aa+Cc+Cd)+B(Ab+Bc+Bd))u,\\
\nabla_{\nabla_{e_1}e_1}V=\nabla_{\nabla_{e_2}e_2}V=\nabla_{\nabla_{e_3}e_3}V=
\nabla_{\nabla_{e_4}e_4}V=0,
\end{array}
\end{eqnarray}
Hence, from \eqref{con3c2} we deduce
\begin{center} 
$ \nabla^*\nabla V=\sum_i \varepsilon_i( \nabla_{e_i}\nabla_{e_i} V-\nabla_{\nabla_{e_i}e_i}V)=0$. \\
\end{center}
Using \eqref{Xc2}, the curvature tensor is completely determined by 
\begin{center} 
$R(\nabla_{e_1} V,V)e_1=R(\nabla_{e_2} V,V)e_2=R(\nabla_{e_3} V,V)e_3=R(\nabla_{e_4} V,V)e_4=0.$\\
\end{center}
Therefore
\begin{center} 
$tr[R(\nabla_. V,V)_.]=\sum_{i} \varepsilon_i R(\nabla_{e_i} V,V)e_i=0$.
\end{center}
A similar argument for case (16) will prove the statement.
\end{proof}
Using \eqref{con3c2}, with regard to geodesic and Killing vector fields we obtain the following.
\begin{prop}\label{proc2}
Let $G$ be a four-dimensional simply connected Lie group of type $\rm(c2)$ and $V\in \g$ a left-invariant vector field on $G$. If $g$ is a left-invariant Lorentzian Einstein metric on $G$, then we have the results in Table 5.
\end{prop}
\begin{table}
\caption{Geodesic, killing and parallel vector fields on $G$: type $\rm(c2)$.\label{tab}}
\begin{tabular}{|c|c|c|c|}
\hline
$(G,g)$ & Geodesic vector fields & Killing vector fields & parallel vector fields \\
\hline
$(15)$& $V=ae_1+be_2+cu$& $V=cu$ & $V=cu$\\
\hline
$(16)$&  $V=a(e_1-e_2)+cu$ 
&$V=a(e_1-e_2)$ & \xmark \\
%
\hline
\end{tabular}
\end{table}
\begin{table}
\caption{Set of equivalent properties for vector fields: type $\rm(c2)$.\label{tab}}
\begin{tabular}{|p{1cm}|p{13cm}|}
\hline
$(G,g)$ &   Equivalent properties  \\
\hline
$(15)$&$V$ is geodesic if and only if $d=-c$; $\hspace*{2mm}\equiv\hspace*{2mm}$  $V$ is harmonic; $\hspace*{2mm}\equiv\hspace*{2mm}$    $V \in \tilde{\mathfrak{X}}^{\rho}(G)$; $\hspace*{2mm}\equiv\hspace*{2mm}$    $V$ defines harmonic map;$\hspace*{2mm}\equiv\hspace*{2mm}$  $V$ is Killing if and only if $a=b=0, d=-c$; $\hspace*{2mm}\equiv\hspace*{2mm}$  $V$ is parallel if and only if $a=b=0, d=-c$; $\hspace*{2mm}\equiv\hspace*{2mm}$    $V=ae_1+be_2+ce_3+de_4$,\\
\hline
$(16)$& $V$ is geodesic if and only if $b=-a$; $\hspace*{2mm}\equiv\hspace*{2mm}$  $V$ is harmonic; $\hspace*{2mm}\equiv\hspace*{2mm}$   $V \in \tilde{\mathfrak{X}}^{\rho}(G)$ $\hspace*{2mm}\equiv\hspace*{2mm}$    $V$ defines harmonic map; $\hspace*{2mm}\equiv\hspace*{2mm}$    $V$ is Killing if and only if $b=-a$ and $c=d=0$;  $\hspace*{2mm}\equiv\hspace*{2mm}$ $V=ae_1+be_2+cu$,\\
\hline
\end{tabular}
\end{table}
Using Proposition \ref{proc2} and Theorem \ref{hor8c2}, a straight forward calculation proves the following classification result.
\begin{cor}\label{hor1c2}
Let $V\in \g$ be a left-invariant vector field on four-dimensional simply connected Lie group $G$ of type $\rm(c2)$. If $g$ is a left-invariant Lorentzian Einstein metric on $G$, then for the
different cases $(15)$ and $(16)$ the equivalent properties for $V$ are listed in the Table 6:
\end{cor}

Clearly, by Remark \ref{spatially}, the results listed in Table 6 show that for cases $(15)$ and $(16)$, $V$ is spatially harmonic if and only if
$d=-c$.

\section{The energy of vector fields}
We calculate explicitly the energy of a vector field $V\in\g$ of a four-dimensional Einstein Lorentzian Lie group. This gives us the opportunity to determine some critical values of the energy functional on four-dimensional Einstein Lorentzian Lie groups.
\subsection{Four-dimensional Einstein Lorentzian Lie group of types $\rm(a1)$ and $\rm(a2)$}
 Let $(G,g)$ be a four-dimensional Einstein Lorentzian Lie group of types $\rm(a1)$ or $\rm(a2)$ and $\lbrace e_i\rbrace _{i=1}^4$ a pseudo-orthonormal basis with $e_3$ time-like. We prove the following now.
\begin{prop}\label{E1}
Let $G$ be a four-dimensional simply connected Lie group of type $\rm(a1)$ or $\rm(a2)$, $V=ae_1+be_2+ce_3+de_4\in \g$ a vector field on $G$ and $\D$ its relatively compact domain. Denote by $E_\D(V)$ the energy of $V|_\D$. For the
different cases $(1)-(11)$ the energy of vector field $V$ is listed in Table 7.
\end{prop}
\begin{table}
\caption{Energy of vector fields: types $\rm(a1)$ and $\rm(a2)$.\label{tab}}
\begin{tabular}{|p{1cm}|p{12.5cm}|}
\hline
$(G,g)$ & $E_\D(V)$ \\
\hline
$(1)$& $(2+A^2((||V||^2+2(d^2-b^2)+2\delta d(b+c)-2bc)/2)vol \D$\\
\hline
$(2)$&$(2+(A+B)^2((a^2+3d^2)(A+B)+(B-A)(b-c)^2-2d(b-c)\sqrt{A^2-B^2})/8)vol \D$\\
\hline
$(3)$&$(2+\frac{(A-B)^2(A+B)^2}{2B^2}||V||^2)vol \D$\\
\hline
$(4)$&$(2+A^2||V||^2/2))vol \D$\\
\hline
$(5)$&$(2+(A+B)(A(a^2-b^2)+2\epsilon bc\sqrt{A^2+AB+B^2}+B(a^2+c^2))/2)vol \D$\\
\hline
$(6)$&$(2+(A^2(4a^2+12d^2+17c^2+24bc+7b^2)/2))vol \D$,  in this case $\epsilon=1$\\
\hline
$(7)$&$(2+(A^2(||V||^2+2d)-B^2(b^2-c^2))/2)vol \D$\\
\hline
$(8)$&$(2+(1/72)(A+B)(A(4a^2-17b^2-7c^2+12d^2-24bc)+B(4a^2+7b^2+17c^2+12d^2+24bc)))vol \D$,  in this case $\epsilon=-1$\\
\hline
$(9)$&$(2+(1/2)A^2(\frac{7}{4}a^2+b^2+\frac{17}{4}c^2+3d^2-6ac))vol \D$, in this case $\epsilon=1$\\
\hline
$(10)$&$(2-(1/2)(AC(a^2-b^2)-C^2(a^2+c^2)-2\epsilon ab\sqrt{-A^2-C^2-AC}C ))vol \D$, in this case $B=0$\\
\hline
$(11)$&$(2+(1/36)A^2(3a^2-b^2+17c^2-5d^2-10ab+9\sqrt{2}ac-9\sqrt{2}bc))vol \D$, in this case $\epsilon=\delta=-1$\\
\hline
\end{tabular}
\end{table}
\begin{proof}
Let $(G,g)$ be a pseudo-Riemannian manifold of dimension $4$. Consider a local pseudo-orthonormal
basis $\lbrace e_1,...,e_n\rbrace$ of vector fields, with $\varepsilon_i= g(e_i,e_i)=\pm 1$ for all indices i. Then, locally,
\begin{center}
$||\nabla V||^2=\sum_{i=1}^{n}\varepsilon_{i} g(\nabla_{e_i}V,\nabla_{e_i}V).$
\end{center}
These conclusions are obtained from a case-by-case argument. As an example,
for case $(3)$, equation \eqref{con2} easily yields
\begin{center}
$||\nabla V||^2=\frac{(A-B)^2(A+B)^2}{B^2}||V||^2.$
\end{center}
Therefore, $||\nabla V||=0$ if and only if $A=\pm B$. Thus, if $A=\pm B$, then vector fields of the same length, will
minimize the energy.
\end{proof}
\subsection{ Four-dimensional Einstein Lorentzian Lie group of type $\rm(c1)$ and $\rm(c2)$}
Let $(G,g)$ be a four-dimensional Einstein Lorentzian Lie group of type $\rm(c1)$ or $\rm(c2)$ and
$\lbrace e_1,...,e_4\rbrace$ a local pseudo-orthonormal
basis of vector fields described in \ref{frame}. We verify the following.
\begin{prop}\label{E2}
Let $G$ be a four-dimensional simply connected Lie group of type $\rm(c1)$ or $\rm(c2)$,  $V=ae_1+be_2+ce_3+de_4\in \g$ a vector field on $G$ and $\D$ its relatively compact domain. Denote by $E_\D(V)$ the energy of $V|_\D$. For the
different cases $(12)-(16)$ the energy  of vector field $V$ is listed in Table 8.
\end{prop}
\begin{table}
\caption{Energy  of vector fields: types $\rm(c1)$ and $\rm(c2)$.\label{tab}}
\begin{tabular}{|p{1cm}|p{11.5cm}|}
\hline
$(G,g)$ & $E_\D(V)$ \\
\hline
$(12)$&$(2+((A+B)^2+C^2)(c+d)^2/2))vol \D$\\
\hline
$(13)$&$(2+(B^2+4A^2-2CB-2BD+(C+D)^2)(B^2+4A^2+2CB+2BD+(C+D)^2)(c+d)^2/32A^2))vol \D$\\
\hline
$(14)$&$(2+((A+D)^2+4B^2)(c+d)^2/2))vol \D$\\
\hline
$(15)$&$2vol \D$\\
\hline
$(16)$&$(2+((A+B)^2+A^2+B^2)(c+d)^2/2))vol \D$\\
\hline
\end{tabular}
\end{table}
\begin{proof}
We followed the same argument used in Proposition \ref{E1}, for case (14).
From equation \eqref{Xc2}, we deduce
\begin{center}
$||\nabla V||^2=((A+D)^2+4B^2)(c+d)^2.$
\end{center}
Thus, $||\nabla V||=0$ if and only if $c=-d$. Therefore, if $c=-d$, then vector fields of the same length, will
minimize the energy.
\end{proof}

\end{document}